\documentclass[reqno]{amsart}

\usepackage{amsmath, amsfonts, amsthm, amssymb, color,  graphicx, mathrsfs, cite}
\usepackage{comment,stmaryrd}

\theoremstyle{plain}
\newtheorem{theorem}{Theorem}[section]
\newtheorem{lemma}{Lemma}[section]
\newtheorem{proposition}{Proposition}[section]
\newtheorem{corollary}{Corollary}[section]

\theoremstyle{definition}
\newtheorem{definition}{Definition}[section]

\theoremstyle{remark}

\newtheorem{example}{Example}[section]

\numberwithin{equation}{section}
\allowdisplaybreaks

\usepackage{ifpdf}
\ifpdf \usepackage[colorlinks=true, citecolor=blue, linkcolor=blue, urlcolor=blue]{hyperref} \fi

\def\thrm{\begin{theorem}}
\def\thrml#1{\begin{theorem}\label{#1}}
\def\ethrm{\end{theorem}}
\def\lmm{\begin{lemma}}
\def\lmml#1{\begin{lemma}\label{#1}}
\def\elmm{\end{lemma}}
\def\dfntn{\begin{definition}}
\def\dfntnl#1{\begin{definition}\label{#1}}
\def\edfntn{\end{definition}}
\def\crllr{\begin{corollary}}
\def\crllrl#1{\begin{corollary}\label{#1}}
\def\ecrllr{\end{corollary}}
\def\xmpl{\begin{example}}
\def\xmpll#1{\begin{example}\label{#1}}
\def\exmpl{\end{example}}
\def\nmrt{\begin{enumerate}}
\def\enmrt{\end{enumerate}}
\def\qtn{\begin{equation}}
\def\qtnl#1{\begin{equation}\label{#1}}
\def\eqtn{\end{equation}}
\def\prpstn{\begin{proposition}}
\def\prpstnl#1{\begin{proposition}\label{#1}}
\def\eprpstn{\end{proposition}}

\def\proof{{\bf Proof}.\ }
\def\eprf{\hfill$\square$}

\def\qaq{\quad\text{and}\quad}

\def\cS{\mathcal {S}}

\def\irr{{\rm Irr}}

\def\qaq{\quad\text{and}\quad}

\begin{document}

\title{An affirmative answer to a question on connectivity of $p$-subgroup posets with irreducible characters}
\maketitle

\begin{center}
{\author Gang Chen$^{a, *}$,  Wenhua Zhao$^{b}$}
\end{center}

\vskip 3mm

\begin{abstract}

 Let $p$ be a prime, $e$ a nonnegative integer, and $G$  a finite $p$-group with  $p^{e+1}\mid|G|$. Let $I$ be the intersection of all subgroups of order $p^{e+1}$ in $G$. It is proved that $|I\cap Z(G)|\le |\pi_0(\Gamma_{p, e}(G))|\le |\irr (I)|$, where $\Gamma_{p,e}$, whose connected components is denoted by $\pi_0(\Gamma_{p, e}(G))$,  is the poset consisting of all pairs $(H, \varphi)$ with $H\le G$, $|H|\ge p^{e+1}$,  $\varphi\in \irr(H)$. Hence, an affirmative answer to Question 2 in \cite{MY} is obtained. 
\vskip 2mm

\end{abstract}

\maketitle
\date{}

\maketitle


\renewcommand{\thefootnote}{\empty}
\footnotetext{$^{a, *}$ Email: {997051@hainanu.edu.cn},  School of Mathematics and Statistics,  Hainan University, Haikou 570228, China}
\footnotetext{ $^{b}$ Corresponding author,  Email:  {zwh1999@mails.ccnu.edu.cn},  School of Mathematics and Statistics, Central China Normal University, Wuhan 430079, China.}

\section{Introduction}

Throughout this paper,  $p$ is a prime and $e$ is  a nonnegative integer. 

\medskip 
Let  $G$ be a finite group. The set of all $p$-subgroups of $G$ whose orders are greater than $p^e$ is denoted by $\cS_{p,e}(G)$. 

\medskip 

In \cite{MY}, the following poset is defined:
$$
\Gamma_{p,e}(G)=\{(H, \varphi)| H\in \cS_{p,e}(G), \varphi\in \irr(H)     \}, 
$$
with the partial order $``\le"$  defined as
$$
(K, \psi)\le (H, \varphi)\, \, \Longleftrightarrow \, \, K\le H \qaq [\varphi_K, \psi]\ne 0. 
$$ 
In addition,  two pairs $(K, \psi)$ and $(H, \varphi)$ in $\Gamma_{p,e}(G)$ are connected or lie in the same connected component if they can be joined by a finite sequence in which every two consecutive pair is compatible. The set of connected components of $\Gamma_{p, e}(G)$ is denoted by $\pi_0(\Gamma_{p,e}(G))$.  

\medskip
For a finite $p$-group $G$ with $p^{e+1}\mid |G|$, the intersection of all subgroups of order $p^{e+1}$ of $G$ is denoted by $I$. 

\medskip 

If $p^2\mid |G|$, it is proved in Theorem C of \cite{MY} that $\Gamma_{p,1}$ is disconnected iff $I\ne 1$. In Question 2 of \cite{MY}, Meng and Yang raise the following question: 

\medskip

{\bf Question.}  Let $G$ be a finite $p$-group, where $|G|$ is divisible by $p^{e+1}$. Assume that ~$e\ge 2$. Is it true that $\Gamma_{p,e}(G)$ is disconnected iff $I\ne 1$?

\medskip 

One implication has been prove at the end of \cite{MY}. In this short note, an affirmative answer to the previous question is provided; see Theorem \ref{2113c} below.

\section{Proofs}

The following lemma is very useful. 

\begin{lemma}\label{2125u}(\cite[Lemma 4]{MY}) Let $G$ be a finite $p$-group and $H$ a subgroup of $G$ with $|H|\ge p^{e+1}$. Then
	$$
	\pi_0(\Gamma_{p,e}(G))=\{[(H, \varphi)]|\varphi\in \irr(H)\}, 
	$$
	where $[(H, \varphi)]$ denotes the connected component containing $(H, \varphi)$ in $\Gamma_{p,e}(G)$. 
\end{lemma}

\medskip

\begin{theorem}\label{3}
Let $H,K\in \cS_{p,e}(G)$,  $\alpha\in \irr(H)$, and $\beta\in \irr(K)$ such that  $[\alpha_{H\cap K},\beta_{H\cap K}]\neq 0$. Then $(H,\alpha)$ is connected with $(K,\beta)$ in $\Gamma_{p,e}(G)$.
\end{theorem}
\proof
By the assumption,  there exists $\gamma\in \irr(H\cap K)$ such that both $\alpha$ and~$\beta$ lie over ~$\gamma$. Thus, 
$$
  [(\alpha^G)_{K},\beta]=\sum\limits_{x\in[H\backslash G/K]}[(^x\alpha)_{^xH\cap K}^K,\beta]\geqslant [\alpha_{H\cap K}^K,\beta]=[\alpha_{H\cap K},\beta_{H\cap K}]\geqslant[\gamma,\gamma]\neq 0, 
$$
where the first equality holds by Mackey Formula (see \cite[Problem 5.6]{I1}) and the second equality holds by Frobenius Reciprocity (see \cite[Lemma 5.2]{I1}.

\medskip

Thus, we may choose an irreducible constituent $\omega$ of $\alpha^G$ such that $[\omega_K, \beta]\ne 0$.  Then one can see that 
$$
(H,\alpha)\leqslant (G,\omega)\geqslant (K,\beta), 
$$
 which proves the assertion. 
\eprf

\medskip

The following theorem consists of the key ingredient of the main results. 

\medskip

\begin{theorem}\label{4}
 Suppose that $L_0,L_1,L_2\ldots L_n, L_{n+1}\in \cS_{p,e}(G)$. For each $j$, set $K_j=\bigcap\limits_{i=0}^jL_i$. Let $\alpha_0\in \irr(L_0)$, $\alpha_{n+1}\in \irr(L_{n+1})$ satisfying $[(\alpha_0)_{K_{n+1}},(\alpha_{n+1})_{K_{n+1}}]\neq 0$.  Then $(L_0,\alpha_0)$ is connected with $(L_{n+1},\alpha_{n+1})$ in $\Gamma_{p,e}(G)$.
\end{theorem}
\proof  We proceed by induction on $n$. If  $n=1$, the assertion follows by Theorem~\ref{3}.   
  
  \medskip

  Now assume that $n>1$. By the assumption, there exists $\gamma\in \irr(K_{n+1})$ such that  both ~$\alpha_0$ and $\alpha_{n+1}$ lie over~$\gamma$.
  
  Since
  $$
  0\ne [(\alpha_{n+1})_{K_{n+1}}, \gamma]=[((\alpha_{n+1})_{L_n\cap L_{n+1}})_{K_{n+1}}, \gamma], 
  $$
  there exists $\eta\in \irr(L_n\cap L_{n+1})$ such that $\eta$ lies over $\gamma$ and lies under $\alpha_{n+1}$.
  
  \medskip
  
  Then we have
  
  \begin{displaymath}
  	\begin{split}
  		&[(\alpha_0)_{K_n},(\eta^{L_n})_{K_n}]=\sum\limits_{x}[(\alpha_0)_{K_n},(^x\eta)_{^x(L_n\cap L_{n+1})\cap K_n}^{K_n}]\ge [(\alpha_0)_{K_n},\eta_{K_{n+1}}^{K_n}]\\
  		=&[(\alpha_0)_{K_{n+1}},\eta_{K_{n+1}}]\ge[\gamma,\gamma]=1, 
  	\end{split}
  \end{displaymath}
  where $x$ runs over a set of $(L_n\cap L_{n+1})\backslash L_n/K_n$ double coset representatives. 
  
  \medskip
  
  Thus there exists $\alpha_n\in \irr(\eta^{L_n})$ such that $[(\alpha_0)_{K_n},(\alpha_n)_{K_n}]\neq 0$. By the inductive hypothesis, we have $(L_0,\alpha_0)$ is connected with $(L_n,\alpha_n)$ in $\Gamma_{p,e}(G)$. 
  
  \medskip 
  
  Finally, note that
  $$
  [(\alpha_n)_{L_{n}\cap L_{n+1}},(\alpha_{n+1})_{L_n\cap L_{n+1}})]\geqslant [\eta,\eta]=1, 
  $$
which yields that  $(L_n,\alpha_{n})$ is connected with $(L_{n+1},\alpha_{n+1})$ by Theorem \ref{3}. 

\medskip 

We conclude that $(L_0, \alpha_0)$ is connected with $(L_{n+1}, \alpha_{n+1})$. 

\eprf

\medskip

The following result will provide an affirmative answer to Question 2 in \cite{MY}. 

\medskip

\begin{theorem}\label{2113c}
  Let  $G$ be a finite ~$p$-group with $p^{e+1}\mid |G|$, and let $I$ be the intersection of all subgroups of order $p^{e+1}$ in $G$. 
  Then
  $$
  |I\cap Z(G)|\leqslant |\pi_0(\Gamma_{p,e}(G))|\leqslant |\irr(I)|.
  $$
  In particular, we have $\Gamma_{p,e}(G)$ is connected iff $I=1$.
\end{theorem}
\proof  Choose subgroups $L_0,  L_1,\ldots, L_n$  such that 
  $$
  \bigcap\limits_{i=0}^{n} L_i=I, 
  $$
  where each $L_i$ has order $p^{e+1}$, $0\le i \le n$. 
  
  \medskip
  
  For any $(L_0, \alpha_0), (L_0, \alpha_0')\in \Gamma_{p,e}(G)$, by taking $(L_{n+1}, \alpha_{n+1})=(L_0, \alpha_0')$  and noting that in the notation of Theorem \ref{4} $K_{n+1}=I$, one can see that, by Theorem ~\ref{4},  $(L_0, \alpha_0)$ is connected with $(L_0, \alpha_0')$ if $(\alpha_0)_I$ and 
  $(\alpha_0')_I$ have a common irreducible constituent. 
  
  As a consequence of Lemma \ref{2125u}, we have
  $$
  |\pi_0(\Gamma_{p,e}(G))|\leqslant |\irr(I)|, 
  $$
  which proves  the righhand side inequality in the theorem. 

  If $I\cap Z(G)=1$, the lefthand side inequality in the theorem holds trivially. So we may assume that  $|I\cap Z(G)|=p^{f+1}$ with $f$ a nonnegative integer. 
  
  \medskip Define  
  $$
  \begin{aligned}
    \varphi: \Gamma_{p,e}(G)&\rightarrow \Gamma_{p,f}(I\cap Z(G))\\
    (H,\alpha)&\mapsto (I\cap Z(G),\beta), 
  \end{aligned}
  $$
  where $\beta$ is the unique constituent of $\alpha_{I\cap Z(G)}$.  
  
  \medskip
  
   Note that if $(H_1,\alpha_1)\leqslant (H_2,\alpha_2)$, we have $\varphi((H_1,\alpha_1))=\varphi((H_2,\alpha_2))$, as $\alpha_1$ is an irreducible constituent of $\alpha_2$ and $(\alpha_2)_{I\cap Z(G)}$ has a unique irreducible constituent. Hence,  $\varphi$ is a poset map. 
  
  \medskip 
   
   It is easy to check $\varphi$ is a surjection. By Lemma 5 of \cite{MY},   we have
   $$
  |\pi_0(\Gamma_{p,f}(I\cap Z(G))|\leqslant |\pi_0(\Gamma_{p,e}(G))|.  
  $$
  
  \medskip
  
  Since $I\cap Z(G)$ is abelian and $|I\cap Z(G)|=p^{f+1}$, one can easily see that 
  $$
  |\pi_0(\Gamma_{p,f}(I\cap Z(G))|=|I\cap Z(G)|, 
  $$
  and hence the lefthand side inequality in the theorem is proved.  
  
  \medskip 
    
  In particular, if $\Gamma_{p,e}(G)$ is connected, then $|\pi_0(\Gamma_{p,e}(G))|=1$, which yields that  $I\cap Z(G)=1$ and therefore $I=1$ as  $I\vartriangleleft G$. Conversely, if $I=1$ then  $|\pi_0(\Gamma_{p,e}(G))|=1$, i.e., $\Gamma_{p,e}(G)$ is connected. 
\eprf

\section{Acknowledgment}

The work of the  authors is supported by Natural Science Foundation of China (No. 12371019).

\end{document}